\documentclass[11pt]{article}
\usepackage{amsmath,amsfonts,amssymb,amsthm,enumerate,graphicx}
\usepackage{tikz,authblk}
\usepackage{tikz-cd}
\usepackage{subfig}
\usepackage{url}
\usepackage{float}
\usepackage{easyReview}

\newtheorem{theorem}{{\bf Theorem}}[section]
\newtheorem{result}{{\bf Result}}[section]
\newtheorem{claim}{{\bf Claim}}[section]

\newtheorem{lemma}[theorem]{{\bf Lemma}}

\newtheorem{proposition}[theorem]{{\bf Proposition}}

\newtheorem{question}[theorem]{{\bf Question}}

\newcommand{\vg}{\vspace{2cm}}

\newcommand{\ol}{\overline}

\newcommand{\ZZ}{ \ensuremath{\mathbb{Z}}}

\begin{document}

\title{Semi\mbox{-}equivelar toroidal maps and their $k$-semiregular covers}

	\author[1] {Arnab Kundu}
 	\author[2] { Dipendu Maity}
 	\affil[1, 2]{Department of Science and Mathematics,
 		Indian Institute of Information Technology Guwahati, Bongora, Assam-781\,015, India.\linebreak
 		\{arnab.kundu, dipendu\}@iiitg.ac.in/\{kunduarnab1998, dipendumaity\}@gmail.com.}

\date{\today}

\maketitle

\begin{abstract}
If the face\mbox{-}cycles at all the vertices in a map are of same type then the map is called semi\mbox{-}equivelar. In particular, it is called equivelar if the face-cycles contain same type of faces. A map is semiregular (or almost regular) if it has as few flag orbits as possible for its type. A map is $k$-regular if it is equivelar and the number of flag orbits of the map $k$ under the automorphism group. In particular, if $k =1$, its called regular. A map is $k$-semiregular if it contains more number of flags as compared to its type with the number of flags orbits $k$. Drach et al. \cite{drach:2019} have proved that every semi-equivelar toroidal map has a finite unique minimal semiregular cover. In this article, we show the bounds of flag orbits of semi-equivelar toroidal maps, i.e., there exists $k$ for each type such that every semi-equivelar map is $\ell$-uniform for some $\ell \le k$. We show that none of the Archimedean types on the torus is semiregular, i.e., for each type, there exists a map whose number of flag orbits is more than its type. We also prove that if a semi-equivelar map is $m$-semiregular then it has a finite index $t$-semiregular minimal cover for $t \le m$. We also show the existence and classification of $n$ sheeted $k$-semiregular maps for some $k$ of semi-equivelar toroidal maps for each $n \in \mathbb{N}$. 
\end{abstract}

\noindent {\small {\em MSC 2010\,:} 52C20, 52B70, 51M20, 57M60.

 }

\section{Introduction}

A {\em map M} is an embedding of a graph {\em G} on a surface {\em S} such that the closure of components of $S \setminus G$, called the {\em faces} of $M$, are homeomorphic to $2$-discs. A map $M$ is said to be a {\em polyhedral map} if the intersection of any two distinct faces is either empty, a common vertex, or a common edge. Here map means a polyhedral map.

The $face\mbox{-}cycle$ $C_u$ of a vertex $u$ (also called the {\em vertex-figure} at $u$) in a map is the ordered sequence of faces incident to $u$.
So, $C_u$ is of the form $(F_{1,1}\mbox{-}\cdots \mbox{-}F_{1,n_1})\mbox{-}\cdots\mbox{-}(F_{k,1}\mbox{-}$ $\cdots \mbox{-}F_{k,n_k})\mbox{-}F_{1,1}$, where $F_{i,\ell}$ is a $p_i$-gon for $1\leq \ell \leq n_i$, $1\leq i \leq k$, $p_r\neq p_{r+1}$ for $1\leq r\leq k-1$ and $p_n\neq p_1$. The types of the faces in $C_u$ defines the type of $C_u$. In this case, the type of face-cycle($u$) is $[p_1^{n_1}, \dots, p_k^{n_k}]$, is called vertex type of $u$. A map $M$ is called {\em semi-equivelar} (\cite{DM2018}, we are including the same definition for the sake of completeness) if $C_u$ and $C_v$ are of same type for all $u, v \in V(X)$. More precisely, there exist integers $p_1, \dots, p_k\geq 3$ and $n_1, \dots, n_k\geq 1$, $p_i\neq p_{i+1}$ (addition in the suffix is modulo $k$) such that $C_u$ is of the form as above for all $u\in V(X)$. In such a case, $X$ is called a semi-equivelar map of type (or vertex type) $[p_1^{n_1}, \dots, p_k^{n_k}]$ (or, a map of type $[p_1^{n_1}, \dots, p_k^{n_k}]$).

Two maps of fixed type on the torus are {\em isomorphic} if there exists a {\em homeomorphism} of the torus which maps vertices to vertices, edges to edges, faces to faces and preserves incidents. More precisely,
if we consider two polyhedral complexes $M_{1}$ and $M_{2}$ then an isomorphism to be a map $f ~:~ M_{1}\rightarrow M_{2}$ such that $f|_{V(M_{1})} : V(M_{1}) \rightarrow V(M_{2})$ is a bijection and $f(\sigma)$ is a cell in $M_{2}$ if and only if $\sigma$ is a cell in $M_{1}$. In particular, if $M_1 = M_2$, then $f$ is called an $automorphism$. The \emph{automorphism group $Aut(M)$} of $M$ is the group consisting of automorphisms of $M$.  

Throughout the last few decades there have been many results about maps and semi-equivelar maps that are highly symmetric. In particular, there has been
recent interest in the study of discrete objects using combinatorial, geometric, and algebraic approaches, with the topic of symmetries of maps receiving a lot of interest. There is a great history of work surrounding maps on the Euclidean plane $\mathbb{R}^2$ and  the $2$-dimensional torus.

An {\em Archimedean} tiling of the plane $\mathbb{R}^2$ is a tiling of $\mathbb{R}^2$ by regular polygons such that all the vertices of the tiling are of same type.
Gr\"{u}nbaum and Shephard \cite{GS1977} showed that there are exactly eleven types of Archimedean tilings on the plane (see Section \ref{sec:examples}). These types are $[3^6]$, $[4^4]$, $[6^3]$, $[3^4,6^1]$, $[3^3,4^2]$,  $[3^2,4^1,3^1,4^1]$, $[3^1,6^1,3^1,6^1]$, $[3^1,4^1,6^1,4^1]$, $[3^1,12^2]$,  $[4^1,6^1,12^1]$, $[4^1,8^2]$. Clearly, these tilings are also semi-equivelar on $\mathbb{R}^2$. But, there are semi-equivelar maps on $\mathbb{R}^2$ which are not (not isomorphic to) Archimedean tilings. In fact, there exists $[p^q]$ equivelar maps on $\mathbb{R}^2$ whenever $1/p+1/q < 1/2$ (e.g., \cite{CM1957}, \cite{FT1965}). We know from \cite{DU2005, DM2017, DM2018}  that the Archimedean tilings $E_i$ $(1 \le i \le 11)$ (in Section \ref{sec:examples}) are unique as semi-equivelar maps. That is, we have the following. 

\begin{proposition} \label{theo:plane}
Let $E_1, \dots, E_{11}$ be the Archimedean tilings on the plane given in Section $\ref{sec:examples}$. Let $X$ be a semi-equivelar map on the plane. If the type of $X$ is same as the type of $E_i$, for some $i\leq 11$, then $X\cong E_i$. In particular, $X$ is vertex-transitive.
\end{proposition}

As a consequence of Proposition \ref{theo:plane} we have 

\begin{proposition} \label{propo-1}
All semi-equivelar maps on the torus are the quotient of an Archimedean tiling on the plane by a discrete subgroup of the automorphism group of the tiling.
\end{proposition}

A map is {\em regular} if its automorphism group acts regularly on flags (which, in nondegenerate cases, may be identified with mutually incident vertex-edge-face triples). In general, a map is \emph{semiregular} (or \emph{almost regular}) if it has as few flag orbits as possible for its type. A map is \emph{$k$-regular} if it is equivelar and the number of flag orbits of the map $k$ under the automorphism group. In particular, if $k =1$, its called regular. Similarly, a map is called \emph{$k$-semiregular} if it contains more number of flags as compared to its type and the number of flags orbits $k$.
The study of regular maps on compact surfaces has a long and rich history. Its early stages go back to the ancient Greeks' interest in highly symmetric solids and (much later) to Kepler's discovery of stellated polyhedra. A new dimension to the combinatorial and group-theoretic nature of the study of highly symmetric maps was added in the late 19th century in the work of Klein and Poincar{\'e} by revealing facts that relate the theory of maps to hyperbolic geometry and automorphic functions.

A systematic approach to classification of regular maps on a given surface was initiated by Brahana in the early 20th century. In the span of the following 70 years this was gradually extended by contributions of numerous authors,
resulting by the end of 1980's in a classification of all chiral and regular maps on orientable surfaces of genus up to 7, and regular maps on nonorientable surfaces of genus at most 8. Details of this development are summarized in
the survey paper \cite{siran2006}.
In 2000, the classification was extended with the help of computing power to orientable and nonorientable surfaces of genus up to 101 and 202, respectively \cite{conder2001}. Nevertheless, by the end of 20th century, classification of regular maps
was available only for a finite number of surfaces. 

Many ideas of the discrete symmetric structures on torus follow from the concepts introduced by Coxeter and Moser in \cite{CM1957}.  A surjective mapping $\eta \colon X \to Y$ from a map $X$ to a map $Y$ is called a $covering$ if it preserves adjacency and sends vertices, edges, faces of $X$ to vertices, edges, faces of $Y$ respectively. That is, let $G \leq$Aut($X$) be a discrete group acting on a map $X$ \emph{properly discontinuously} (\cite[Chapter 2]{katok:1992}). This means that each element $g$ of $G$ is associated with an automorphism $h_g$ of $X$ onto itself, in such a way that $h_{gh}$ is always equal to $h_g h_h$ for any two elements $g$ and $h$ of $G$, and $G$-orbit of any vertex $u\in V(X)$ is locally finite. Then, there exists $\Gamma \leq $Aut($X$) such that $Y = X/\Gamma$. In such a case, $X$ is called a cover of $Y$. A map $X$ is called regular if the automorphism group of $X$ acts transitively on the set of flags of $X$. Clearly, if a semi-equivelar map is not equivelar then it cannot be regular. 

A natural question then is: 

\begin{question}\label{ques}
Let $X$ be a semi-equivelar  map on the torus. Let $X$ be $k$-semiregular. Does there exist any cover $Y(\neq X)$ of some $m$-semiregular map? Does this cover exist for every sheet, if so, how many? How the flag orbits of $X$ and $Y$ are related? 
\end{question}

In this context, there is also much interest in finding minimal regular covers of different families of maps and polytopes (see \cite{HW2012, MPW2013, pw2011}). In \cite{drach:2015}, Drach et al. constructed the minimal rotary cover of any equivelar toroidal map. Then, they have extended their idea to toroidal maps that are no longer equivelar, and constructed minimal toroidal covers of the Archimedean toroidal maps with maximal symmetry (see in \cite{drach:2019}), called these covers almost regular; they will no longer be regular (or chiral), but instead will have the same number of flag orbits as their associated tessellation of the Euclidean plane. Here, we have the following. 
 
\begin{theorem} \label{no-of-orbits}
Let $X$ be a semi-equivelar map on the torus. Let the flags of $X$ form $m$ ${\rm Aut}(X)$-orbits.\\
{\rm (a)} If the type of $X$ is  $[3^6]$ or $[6^3]$ then $m \leq 6$.\\ 
{\rm (b)} If the type of $X$ is  $[4^4]$  then $m \leq 4$.\\ 
{\rm (c)} If the type of $X$ is  $[3^3, 4^2]$ or $[3^2, 4^1, 3^1, 4^1]$ then $m \leq 10$.\\ 
{\rm (d)} If the type of $X$ is  $[4^1, 8^2]$ or $[3^1, 6^1, 3^1, 6^1]$ then $m \leq 12$.\\ 
{\rm (e)} If the type of $X$ is  $[3^1, 12^2]$  then $m \leq 18$.\\ 
{\rm (f)} If the type of $X$ is  $[3^1, 4^1, 6^1, 4^1]$  then $m \leq 24$.\\ 
{\rm (g)} If the type of $X$ is  $[3^4, 6^1]$  then $m \leq 30$.\\ 
{\rm (h)} If the type of $X$ is  $[4^1, 6^1, 12^1]$  then $m \leq 36$.
These bounds are also sharp. 
\end{theorem}

\begin{proposition} (\cite{drach:2015, drach:2019}) \label{drach2019}
Let $E$ be an Archimedean tiling of type $Z$ and $k$-semiregular. If $X$ is a semi\mbox{-}equivelar toroidal map of type $Z$ then there exists a covering $\eta \colon Y \to X$ where $Y$ is $k$-semiregular and unique.
\end{proposition}

In this context of Prop. \ref{drach2019}, we prove the following. 

\begin{theorem}\label{thm-main1}
    {\rm (a)} If $X_1$ is a $m_1$-semiregular toroidal map of type $[3^6]$ or $[6^3]$, then there exists a covering $\eta_{k_1} \colon Y_{k_1} \to X_1$ where $Y_{k_1}$ is $k_1$-semiregular for each $k_1 \le m_1$ except $k_1 = 4,5$.\\
    {\rm (b)} If $X_2$ is a $m_2$-semiregular toroidal map of type $[4^4]$, then there exists a covering $\eta_{k_2} \colon Y_{k_2} \to X_2$ where $Y_{k_2}$ is $k_2$-semiregular for each $k_2 \le m_2$ except $k_2 = 3$.\\
    {\rm (c)} If $X_9$ is a $m_9$-semiregular toroidal map of type $[3^1, 4^1,6^1,4^1]$. Then, there exists a covering $\eta_{k_9} \colon Y_{k_9} \to X_9$ where $Y_{k_9}$ is $k_9$-semiregular for each $(k_9,m_9)=(4,8),(8,24),(4,12),(12,24),$ $(4,24)$.\\
    {\rm (d)} If $X_7$ is a $m_7$-semiregular toroidal map of type $[3^1,6^1,3^1,6^1]$ then there exists a covering $\eta_{k_7} \colon Y_{k_7} \to X_7$ where $Y_{k_7}$ is $k_7$-semiregular for each $k_7 \le m_7$ for $(k_7,m_7)=(4,8),(8,24),(4,12),$ $(12,24),(4,24)$\\
    {\rm (e)} If $X_8$ is a $m_8$-semiregular toroidal map of type $[3^1,12^2]$  then there exists a covering $\eta_{k_8} \colon Y_{k_8} \to X_8$ where $Y_{k_8}$ is $k_8$-semiregular for each $(k_8,m_8)=(3,6),(3,9),(3,18),(6,18),(9,18)$.\\
    {\rm (f)} If $X_{11}$ is a $m_{11}$-semiregular toroidal map of type $[4^1,6^1,12^1]$, then there exists a covering $\eta_{k_{11}} \colon Y_{k_{11}} \to X_{11}$ where $Y_{k_{11}}$ is $k_{11}$-semiregular for each $k_{11} \le m_{11}$ for $(k_{11},m_{11})=(6,12),(6,18),(6,36),(12,36),(18,36)$.\\
    {\rm (g)} If $X_6$ is a $m_6$-semiregular toroidal map of type  $[4^1,8^2]$ then there exists a covering $\eta_{k_6} \colon Y_{k_6} \to X_6$ where $Y_{k_6}$ is $k_6$-semiregular for each $(k_6,m_6)=(6,12),(3,6),(3,12)$.\\
    {\rm (h)} If $X_4$ is a $m_4$-semiregular toroidal map of type $[3^3,4^2]$, then there exists a covering $\eta_{k_4} \colon Y_{k_4} \to X_4$ where $Y_{k_4}$ is $k_4$-semiregular for $(k_4,m_4)=(5,10)$.\\
    {\rm (i)} If $X_5$ is a $m_5$-semiregular toroidal map of type $[3^2,4^1,3^1, 4^1]$, then there exists a covering $\eta_{k_5} \colon Y_{k_5} \to X_5$ where $Y_{k_5}$ is $k_5$-semiregular for each $(k_5,m_5) = (5,10),(5,20),(10,20)$.\\
    {\rm (j)} If $X_{10}$ is a $m_{10}$-semiregular toroidal map of type $[3^4,6^1]$, then there exists a covering $\eta_{k_{10}} \colon Y_{k_{10}} \to X_{10}$ where $Y_{k_{10}}$ is $k_{10}$-semiregular for $(k_{10},m_{10})=(10,30)$.
\end{theorem}

\begin{theorem}\label{thm-main2}
Let $X$ be a semi\mbox{-}equivelar toroidal map and $k$-semiregular. Then, there exists a $n$ sheeted covering $\eta \colon Y \to X$ for each $n \in \mathbb{N}$ where $Y$ is $m$-semiregular for some $m\le k$.
\end{theorem}

\begin{theorem}\label{thm-main3}
Let $X$ be a $n$ sheeted semi\mbox{-}equivelar $k$-semiregular toroidal map and $\sigma(n) = \sum_{d|n}d$. Then, there exists different $n$ sheeted $m$-semiregular covering $\eta_{\ell} \colon Y_{\ell} \to X$ for $ \ell \in \{1, 2, \dots, \sigma(n)\}$, i.e., $Y_{1},$ $Y_{2},$ $\dots,$ $Y_{\sigma(n)}$ are $n$ sheeted $m$-semiregular covers of $X$ and different upto isomorphism for some $m\le k$.
\end{theorem}

\begin{theorem}\label{thm-main4}
Let $X$ be a $m$-semiregular semi\mbox{-}equivelar toroidal map and $Y$ be a $k$-semiregular covers of $X$. Then, there exists a $k$-semiregular covering map $\eta \colon Z \to X$ such that $Z$ is minimal.
\end{theorem}


\section{Examples} \label{sec:examples}

We first present eleven Archimedean tilings on the plane. We need these examples for the proofs of our results in Section \ref{sec:proofs-1}.


\begin{figure}[H]
\centering
\begin{minipage}[b][5cm][s]{.45\textwidth}
\centering
\vfill

\vfill
\caption{$E_1$~($[3^6]$)}\label{fig-E_1}
\vspace{\baselineskip}
\end{minipage}\qquad
\begin{minipage}[b][5cm][s]{.45\textwidth}
\centering
\vfill
%
\caption{$E_2$~($[4^4]$)}\label{fig-E_2}
\vfill
\end{minipage}
\end{figure}
\vspace{0.5cm}
\begin{figure}[H]
\centering
\begin{minipage}[b][5cm][s]{.45\textwidth}
\centering
\vfill
%
\vfill
\caption{$E_9$~($[3^1,4^1,6^1,4^1]$)}\label{fig-E_9}
\vspace{\baselineskip}
\end{minipage}\qquad
\begin{minipage}[b][5cm][s]{.45\textwidth}
\centering
\vfill
%
\caption{$E_8$~($[3^1,12^2]$)}\label{fig-E_8}
\vfill
\end{minipage}
\end{figure}
\vg
\begin{figure}[H]
\centering
\begin{minipage}[b][5cm][s]{.45\textwidth}
\centering
\vfill
%
\vfill
\vspace{-1.5cm}
\caption{$E_5$~($[3^2,4^1,3^1,4^1]$)}\label{fig-E_5}
\vspace{\baselineskip}
\end{minipage}\qquad
\begin{minipage}[b][5cm][s]{.45\textwidth}
\centering
\vfill
%
\caption{$E_{7}$~($[3^1,6^1,3^1,6^1]$)}\label{fig-E_7}
\vfill
\end{minipage}
\end{figure}
\vg
\begin{figure}[H]
\centering
\begin{minipage}[b][5cm][s]{.45\textwidth}
\centering
\vfill
%
\vfill
\caption{$E_{10}$~($[3^4,6^1]$)}\label{fig-E_10}
\vspace{\baselineskip}
\end{minipage}\qquad
\begin{minipage}[b][5cm][s]{.44\textwidth}
\centering
\vfill
%
\caption{$E_{11}$~($[4^1,6^1,12^1]$)}\label{fig-E_11}
\vfill
\end{minipage}
\end{figure}
\vg
\begin{figure}[H]
\centering
\begin{minipage}[b][5cm][s]{.45\textwidth}
\centering
\vfill
%
\vfill
\caption{$E_6$~($[4^1,8^2]$)}\label{fig-E_6}
\vspace{\baselineskip}
\end{minipage}\qquad
\begin{minipage}[b][5cm][s]{.45\textwidth}
\centering
\vfill
%
\caption{$E_3$~($[6^3]$)}\label{fig-E_3}
\vfill
\end{minipage}
\end{figure}
\vg
\begin{figure}[H]
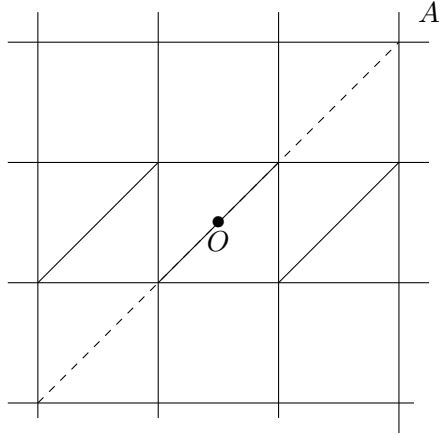

\centering
\begin{minipage}[b][5cm][s]{.45\textwidth}
\centering
\vfill
%
\caption{$E_4$~($[3^3,4^2]$)}\label{fig-E_4}
\vfill
\end{minipage}
\end{figure}
\vspace{1cm}

\section{Classification of $k$-semiregular covers of semi-equivelar maps}\label{sec:proofs-1}

Before going to the proofs of main theorems we need following series of results.
From  \cite[Proposition 3.2-3.7]{drach:2019} we get
\begin{proposition}\label{uni}
Let $E$ be a semi\mbox{-}equivelar tiling on the plane. Suppose $E$ has $m$ flag\mbox{-}orbits. Then 
{\rm (a)} If the type of $E$ is  $[3^6]$, $[4^4]$ or $[6^3]$ then $m = 1$.\\
{\rm (b)} If the type of $E$ is  $[3^1,6^1,3^1,6^1]$ then $m = 2$.\\
{\rm (c)} If the type of $E$ is  $[3^1, 12^2]$ or $[4^1,8^2]$  then $m = 3$.\\
{\rm (d)} If the type of $E$ is  $[3^1,4^1,6^1,4^1]$ then $m = 4$.\\
{\rm (e)} If the type of $E$ is  $ [3^2,4^1,3^1,4^1]$  then $m = 5$. \\
{\rm (f)} If the type of $E$ is  $[4^1,6^1, 12^1]$ then $m = 6$.\\
{\rm (g)} If the type of $E$ is  $[3^4,6^1]$ then $m = 10$.\\
{\rm (h)} If the type of $E$ is  $[3^3,4^2]$ then $m = 5$.
\end{proposition}

\begin{proof}[Proof of Theorem \ref{no-of-orbits}]
Let for $i=1,2,\dots 11$ $E_i$ be the Archimedean tiling of the plane as in Section \ref{sec:examples}. Consider $\alpha_i$ and $\beta_i$ be the fundamental translations of $E_i$. $\alpha_i:z\mapsto z+A_i$ and $\beta_i:z\mapsto z+B_i$. Let $X$ be a semi\mbox{-}equivelar map of type $[p_1^{r_1},\dots p_k^{r_k}]$. Then there exists a discrete subgroup $K_i$ of Aut($E_i$) with out any fixed element such that $X=E_i/K_i$. Let $p_i:E_i \to X$ be the polyhedral covering map.
By above description of $K_i$, it contains only translations and glide reflections. Since, $X$ is orientable so $K_i$ does not contain any glide reflections. Thus $K_i \le H_i$. Suppose $K_i = \langle \gamma_i, \delta_i \rangle$.
Let $\chi_i$ denotes the reflection about origin in $E_i$. Then $\chi_i \in {\rm Aut}(E_i)$. Consider the group $G_i = \langle \alpha_i, \beta_i, \chi_i \rangle \le {\rm Aut}(E_i)$. 
\begin{claim}
$K_i \trianglelefteq G_i$.
\end{claim}
To prove this it is enough to show that $\chi_i\circ\gamma_i \circ \chi_i^{-1}$ and $\chi_i\circ\delta_i \circ \chi_i^{-1} \in K_i$. We know that conjugation of a translation by reflection is translation by the reflected vector.
Let $\gamma_i$ and $\delta_i$ are translation by vectors $C_i$ and $D_i$ respectively. Then $\chi_i\circ\gamma_i \circ \chi_i^{-1}$ and $\chi_i\circ\delta_i \circ \chi_i^{-1}$ are translation by $-C_i$ and $-D_i$. Clearly these vectors are belongs to lattice of $K_i$. Our claim follows from this.

\smallskip
Case 1. Let $X$ is of type $[3^6], [6^3]$. Suppose $X =E_i/K_i$. $E_1$ has $12$ flag orbits by action of $H_1$. Under action of $G_1$, $E_1$ has $6$ flag orbits. Hence action of $G_i/K_i$ on flags of $X$ also gives same number of orbits. $G_i/K_i \le {\rm Aut}(X)$. Thus number of flag orbits of $X$ is less that or equals to $6$ for maps of type $[3^6]$ and $[6^3]$. This proves part (a) of Theorem \ref{no-of-orbits}.  

\smallskip
Case 2. Let $X$ be a semi-equivelar map of type $[4^4]$. Then by Proposition \ref{propo-1} we can assume $X = E_{2}/K_{2}$ for some subgroup $K_{2}$ of Aut($E_{2}$). Now $F(E_{2})$ has $4$ $G_{2}$ orbits. Hence $X$ also has $4$ $G_{2}/K_{2}\mbox{-}$orbits. As $G_{2}/K_{2} \le {\rm Aut}(X)$. Therefore number of ${\rm Aut}(X)\mbox-$orbits of $F(X)$ is less than or equals to $4$.

\smallskip
Case 3. Let $X$ be a semi-equivelar map of type $ [3^3,4^2]$ or $[3^2,4^1,3^1,4^1]$. Then by Proposition \ref{propo-1} we can assume $X = E_{11}/K_{11}$ or $E_5/K_5$ for some subgroup $K_{11}$ of Aut($E_{11}$) and $K_5$ of Aut($E_5$). Now $F(E_{11})$ and $F(E_5)$ has $10$ $G_{11}$ and $G_5$ orbits respectively. Hence $X$ also has $10$ $G_{i}/K_{i}\mbox{-}$orbits for $i=11,5$. As $G_{i}/K_{i} \le {\rm Aut}(X)$. Therefore number of ${\rm Aut}(X)\mbox-$orbits of $F(X)$ is less than or equals to $10$.

\smallskip
Case 4. Let $X$ be a semi-equivelar map of type $[3^1,6^1,3^1,6^1]$ or $[4^1,8^2]$. Then by Proposition \ref{propo-1} we can assume $X = E_{i}/K_{i}$ for some subgroup $K_{i}$ of Aut($E_{i}$) for $i=6,7$. Now $F(E_{i})$ has $12$ $G_{i}$ orbits. Hence $X$ also has $12$ $G_{i}/K_{i}\mbox{-}$orbits. As $G_{i}/K_{i} \le {\rm Aut}(X)$. Therefore number of ${\rm Aut}(X)\mbox-$orbits of $F(X)$ is less than or equals to $12$.

\smallskip
Case 5. Let $X$ be a semi-equivelar map of type $[3^1,12^2]$. Then by Proposition \ref{propo-1} we can assume $X = E_{8}/K_{8}$ for some subgroup $K_{8}$ of Aut($E_{8}$). Now $F(E_{8})$ has $18$ $G_{8}$ orbits. Hence $X$ also has $18$ $G_{8}/K_{8}\mbox{-}$orbits. As $G_{8}/K_{8} \le {\rm Aut}(X)$. Therefore number of ${\rm Aut}(X)\mbox-$orbits of $F(X)$ is less than or equals to $18$.

\smallskip
Case 6. Let $X$ be a semi-equivelar map of type $[3^1,4^1,6^1,4^1]$. Then by Proposition \ref{propo-1} we can assume $X = E_{9}/K_{9}$ for some subgroup $K_{9}$ of Aut($E_{9}$). Now $F(E_{9})$ has $24$ $G_{9}$ orbits. Hence $X$ also has $24$ $G_{9}/K_{9}\mbox{-}$orbits. As $G_{9}/K_{9} \le {\rm Aut}(X)$. Therefore number of ${\rm Aut}(X)\mbox-$orbits of $F(X)$ is less than or equals to $24$.

\smallskip
Case 7. Let $X$ be a semi-equivelar map of type $[3^4,6^1]$. Then by Proposition \ref{propo-1} we can assume $X = E_{10}/K_{10}$ for some subgroup $K_{10}$ of Aut($E_{10}$). Now $F(E_{10})$ has $30$ $G_{10}$ orbits. Hence $X$ also has $30$ $G_{10}/K_{10}\mbox{-}$orbits. As $G_{10}/K_{10} \le {\rm Aut}(X)$. Therefore number of ${\rm Aut}(X)\mbox-$orbits of $F(X)$ is less than or equals to $30$.

\smallskip
Case 8. Let $X$ be a semi-equivelar map of type $[4^1,6^1,12^1]$. Then by Proposition \ref{propo-1} we can assume $X = E_{11}/K_{11}$ for some subgroup $K_{11}$ of Aut($E_{11}$). Now $F(E_{11})$ has $36$ $G_{11}$ orbits. Hence $X$ also has $36$ $G_{11}/K_{11}\mbox{-}$orbits. As $G_{11}/K_{11} \le {\rm Aut}(X)$. Therefore number of ${\rm Aut}(X)\mbox-$orbits of $F(X)$ is less than or equals to $36$.
\end{proof}

\begin{proof}[Proof of Theorem \ref{thm-main1}]
Let $X_1$ be a semi-equivelar toroidal map of type $[3^6]$. Then by Proposition \ref{propo-1} we can assume that $X_1 = E_1/K_1$ for some fixed point free subgroup $K_1$ of Aut($X_1$). Thus $K_1$ consist of only translations and glide reflections. Since $X_1$ is orientable so $K_1$ contains only translations. Consider $H_1$, $K_1$ and $G_1$ as in proof of Theorem \ref{no-of-orbits}. Now $F(X_1)$ has $6$ $G_1\mbox{-}$orbits. Consider the group $G_1' = \langle \alpha_1, \beta_1, \chi_1, \rho_{OA} \rangle$. Clearly $F(E_1)$ has $3$ $G_1'$ orbits. Let $K_1 = \langle \gamma, \delta \rangle$. Now to get a cover of $X$ we need $L_1 \le {\rm Aut}(E_1)$ such that $G_1'/L_1$ is defined and $E_1/L_1$ has $3$ $G_1'/L_1$ orbits. For that we made the following.   
\begin{claim}\label{clm1}
There exists $m \in \ZZ$ such that $L_1 := \langle \gamma^m, \delta^m \rangle \trianglelefteq G_1'.$
\end{claim}
Suppose $L_1 = \langle \gamma^{m_1}, \delta^{m_2} \rangle$. We show that there exists suitable $m_1, m_2$ such that $L_1 \trianglelefteq G_1'$. It turns out that we can take $m_1 = m_2$.
To satisfy $L_1 \trianglelefteq G_1'$ it is enough to show that $\rho_1\gamma^{m_1}\rho_1^{-1}, \rho_1\delta^{m_2}\rho_1^{-1} \in L_1$. It is known that conjugation of a translation by rotation or reflection is also a translation by rotated or reflected vector. Since $\gamma$ and $\delta$ are translation by vectors $C$ and $D$ respectively so $\gamma^{m_1}$ and $\delta^{m_2}$ are translation by vectors $m_1C$ and $m_2D$ respectively. Hence $\rho_1\gamma^{m_1}\rho_1^{-1}$ and $\rho_1\delta^{m_2}\rho_1^{-1}$ are translation by the vectors $C'$ and $D'$ respectively. Where $C' = \rho_1(m_1C) = \rho_1(m_1(aA_1 + bB_1)) = m_1a\rho_1(A_1) + m_1b\rho(B_1) = m_1aB_1 + m_1bA_1$ and similarly $D' = \rho_1(m_2D) = m_2cB_1 + m_2dA_1$.
Now these translations belong to $L_1$ if the vectors $C'$ and $D'$ belong to lattice of $L_1 = \ZZ(m_1C) + \ZZ(m_2D)$.
Let $C', D' \in \ZZ(m_1C) + \ZZ(m_2D)$. Then $\exists~ p, q, s, t \in \ZZ$ such that 
$$C' = p(m_1C) + q(m_2D), ~ D' = s(m_1C) + t(m_2D).$$
Putting expressions of $C', D', C, D$ in above equations we get,
$$(bm_1 - pam_1 - qcm_2)A_1 + (am_1 - pbm_1 - qdm_2)B_1 = 0$$
$$(cm_2 - sbm_1 - tdm_2)A_1 + (dm_2 - sam_1 - tcm_2)B_1 = 0.$$
Since, $\{A_1, B_1\}$ is a linearly independent set we have,
$$pam_1 + qcm_2 = bm_1, ~ pbm_1 + qdm_2 = am_1, ~sbm_1 + tdm_2 = cm_2, ~ sam_1 + tcm_2 = dm_2.$$
Now as rank($L_1$)$=2$ so $m_1, m_2 \neq 0$. Dividing the above system by $m_1m_2$ we get,
$$\frac{pa}{m_2} + \frac{qc}{m_1} = \frac{b}{m_2}, ~\frac{pb}{m_2} + \frac{qd}{m_1} = \frac{a}{m_2}, ~\frac{sb}{m_2} + \frac{td}{m_1} = \frac{c}{m_1}, ~\frac{sa}{m_2} + \frac{tc}{m_1} = \frac{d}{m_1}. $$
Now consider $p, q, s, t$ as variables. We can treat above system as a system of linear equations.  We can write this system in matrix form as follows.
$$\begin{bmatrix} a/m_2 & c/m_1 &0&0\\ 
b/m_2 & d/m_1 &0 &0 \\
0& 0& a/m_2 &c/m_1 \\
0&0& b/m_2&d/m_1\end{bmatrix} 
\begin{bmatrix}
p\\q\\s\\t
\end{bmatrix} =
\begin{bmatrix}
b/m_2\\a/m_2\\d/m_1\\c/m_1
\end{bmatrix}$$
Now $C,D$ are linearly independent thus $ad-bc \neq 0$. Hence the coefficient matrix of the above system has non zero determinant. Therefore the system has an unique solution. After solving we get,
$$p = \frac{m_1^2m_2(d-a)b}{(ad-bc)^2}
,~~q = \frac{m_1^2m_2(a^2-bc)}{(ad-bc)^2},~~s = \frac{m_1m_2^2(d^2-bc)}{(ad-bc)^2}
,~~t = \frac{m_1m_2^2(a-d)c}{(ad-bc)^2}.$$ 
Now if we take $m_1 = m_2  = |ad-bc| = m$(say) then $p, q, s, t \in \ZZ$.
Let $L_1 := \langle \gamma^m, \delta^m \rangle$. Then we have $L_1 \trianglelefteq G_1'.$ Hence our Claim \ref{clm1} proved.
\begin{claim}\label{normal}
$K_1/L_1 \trianglelefteq {\rm Aut}(M_1/L_1)$.
\end{claim}
Let $\rho \in {\rm Nor_{Aut(M_1)}}(L_1)$. Then $\rho \gamma^m \rho^{-1}, \rho \delta^m \rho^{-1} \in L_1$. $\rho \gamma^m \rho^{-1}$ is translation by the vector $(\rho \gamma^m \rho^{-1})(0)$. That is 
$(\rho \gamma^m \rho^{-1})(0) \in$ lattice of $L_1$. Thus there exists $n_1, n_2$ such that $(\rho \gamma^m \rho^{-1})(0) = n_1\gamma^m(0) + n_2 \delta^m(0)$. As $K_1$ is generated by $\gamma$ and $\delta$ so $\rho \gamma \rho^{-1}(0) = n_1\gamma(0) + n_2 \delta(0)$. Thus $\rho \gamma \rho^{-1} \in K_1.$ similarly, $\rho \delta \rho^{-1} \in K_1.$ Therefore $\rho \in {\rm Nor}(K_1) \implies {\rm Nor_{Aut(M_1)}}(L_1) \le {\rm Nor_{Aut(M_1)}}(K_1) \implies K_1 \trianglelefteq {\rm Nor_{Aut(M_1)}}(L_1) \implies K_1/L_1 \trianglelefteq {\rm Nor_{Aut(M_1)}}(L_1)/L_1 $. From \cite{drach:2019} we know ${\rm Aut}(M_1/L_1) = {\rm Nor_{Aut(M_1)}}(L_1)/L_1 .$ This proves  Claim \ref{normal}.

Now by Claim \ref{clm1} $G_1'/L_1$ is a group and acts on $E(M_1/L_1).$ Clearly $E(M_1/L_1)$ has $2~~ G_1'/L_1\mbox{-}$orbits. Since $L_1$ contains two independent vectors, it follows that $Y_1 := M_1/L_1 $ is a toroidal map and $v+L_1 \mapsto v+K_1$ is a covering $\eta : Y_1 \to X$.
Our next aim is to show that $Y$ is a $2$-orbital map. For that we need the following,
\begin{result}\label{quo}
	Let $L \trianglelefteq K$ and $K$ acts on a topological space $E$. Then $\dfrac{E/L}{K/L}$ is homeomorphic to $E/K$ and $\phi:\dfrac{E/L}{K/L} \to E/K$ defined by $\Big(\dfrac{K}{L}\Big)(Lv) \mapsto Kv ~\forall~ v \in E$ is a homeomorphism.
\end{result}
Let $p:M_1/L_1 \to \dfrac{M_1/L_1}{K_1/L_1} = M_1/K_1$ be the quotient map.

\begin{claim}\label{diag}
Given $\alpha \in {\rm Aut}(M_1/L_1)={\rm Nor_{Aut(M_1)}}(L_1)/L_1$ there exists $\widetilde{\alpha} \in {\rm Aut}(M_1/K_1)$ such that $p\circ \alpha = \widetilde{\alpha} \circ p.$
\end{claim}

By Result \ref{quo} we can think $M_1/K_1$ as $\dfrac{M_1/L_1}{K_1/L_1}.$ We show that $\alpha$ takes orbits to orbits for the action of $K_1/L_1$ on $M_1/L_1$.
Let $\mathcal{O}(\ol{v})$ denotes $K_1/L_1$-orbit of $\ol{v} \in M_1/L_1.$ Then
\begin{equation*}
    \begin{split}
        \alpha(\mathcal{O}(\ol{v})) &= \alpha\left(\frac{K_1}{L_1}(\ol{v})\right)  \\&=
        \frac{K_1}{L_1}(\alpha(\ol{v}))~[{\rm since }~ \alpha \frac{K_1}{L_1} = \frac{K_1}{L_1} \alpha~ {\rm because}~ K_1/L_1 \trianglelefteq {\rm Aut(}M_1/L_1{\rm )}]  \\& = \mathcal{O}(\alpha(\ol{v}))
        \end{split}
\end{equation*}
Therefore by universal property of quotient there exists $\widetilde{\alpha} : M_1/K_1 \to M_1/K_1$ such that the following diagram commutes.

\begin{figure}[h]
\begin{center}
\begin{tikzcd}
M_1/L_1 \arrow[r, "\alpha"] \arrow[d, "p"]                                      & M_1/L_1 \arrow[d, "p"]                      \\
\frac{M_1/L_1}{K_1/L_1} \arrow[r, "\widetilde{\alpha}", dashed] & \frac{M_1/L_1}{K_1/L_1}

\end{tikzcd}

\caption{{Diagram }}\label{diagram1}
\end{center}
\end{figure}
Now we have to show that $\widetilde{\alpha}$ is an automorphism.
Clearly $\widetilde{\alpha}$ is onto. Let $\ol{v_1}, \ol{v_2} \in Y.$ Suppose $\mathcal{O}(\ol{v_1})$ and $\mathcal{O}(\ol{v_2})$
be $K_1/L_1$ orbits of $Y$. Now,
\begin{equation*}
    \begin{split}
        \widetilde{\alpha}(\mathcal{O}(\ol{v_1})) = \widetilde{\alpha}(\mathcal{O}(\ol{v_2})) & \implies \mathcal{O}(\alpha(\ol{v_1})) = \mathcal{O}(\alpha(\ol{v_2})) \\ & \implies \exists ~\omega \in K_1/L_1 ~{\rm such~ that }~ \omega\alpha(\ol{v_1}) = \alpha(\ol{v_2}) \\&\implies \alpha^{-1}\omega \alpha (\ol{v_1}) = \ol{v_2} \\& \implies \mathcal{O}(\ol{v_1}) = \mathcal{O}(\ol{v_2}) ~[\alpha^{-1}\omega \alpha \in K_1/L_1 ~{\rm since}~ K_1/L_1 \trianglelefteq {\rm Aut}(M_1/L_1)]
    \end{split}
\end{equation*}
Therefore, $\widetilde{\alpha}$ is one-one. Now, by the commutativity of the diagram and using the fact that $p$ is a covering map one can see that $\widetilde{\alpha}$ takes vertices to vertices, edges to edges, faces to faces. It also preserves incidence relations. Let $v\in $ Domain of $\widetilde{\alpha}$. Since $p$ is a covering map there exists a neighbourhood $N$ of $v$ which is evenly covered by $p$. Let $U$ be a component of $p^{-1}(N)$. Then $p:U\to N$ is a homeomorphism. Therefore $(p \circ \alpha)|_U = \widetilde{\alpha}|_N$. As $p$ and $\alpha$ both are continuous so is $\widetilde{\alpha}|_N$. Thus $\widetilde{\alpha}$ is continuous. 
Now, replacing $\alpha$ by $\alpha^{-1}$ we get $\widetilde{\beta}$ in place of $\widetilde{\alpha}$. $\widetilde{\beta}$ has same properties as of $\widetilde{\alpha}$. Now, $\widetilde{\alpha} \circ \widetilde{\beta} = id_{M_1/K_1} = \widetilde{\beta} \circ \widetilde{\alpha}$. Therefore $\widetilde{\alpha}^{-1} = \widetilde{\beta}$. So $\widetilde{\alpha}$ is a homeomorphism.
Hence $\widetilde{\alpha}$ is an automorphism of $M_1/K_1$.

\begin{claim}\label{clm5}
If $\alpha \in$ Aut$(M_1/L_1) \setminus \frac{G_1'}{L_1}$ then $\alpha(\mathcal{O}) = \mathcal{O}$ for all $\frac{G_1'}{L_1}$-orbits $\mathcal{O}$ of $M_1/L_1$.
\end{claim}

Let $\alpha \in $Aut$(M_1/L_1) \setminus \frac{G_1'}{L_1}$ and $\widetilde{\alpha}$ be the induced automorphism on $M_1/K_1$ as in Fig. \ref{diagram1}. Suppose, $\mathcal{O}_1$ and $\mathcal{O}_2$ be two $G_1'/L_1$-orbits of $M_1/L_1$. Let $a_1, a_2 \in M_1$ be such that $L_1a_1 \in \mathcal{O}_1$ and $L_1a_2 \in \mathcal{O}_2$ and $\alpha(L_1a_1) = L_1a_2$.
Since, $p(L_1a_i) = K_1a_i$ by commutativity of the diagram in Claim \ref{diag} we get $\widetilde{\alpha}(K_1a_1) = K_1a_2$. 
As $\widetilde\alpha$ does not take an element of $G_1/K_1$-orbit to an element of some other orbit so $K_1a_1$ and $K_1a_2$ belong to same $G_1/K_1$-orbit of $M_1/K_1$.
Therefore, there exists $gK_1 \in G_1/K_1$ such that $(gK_1)(K_1a_1) = K_1a_2 $.\\
Now, Since $(gK_1)(K_1a_1) = K_1(ga_1)$ thus $(gK_1)(K_1a_1) = K_1a_2 \implies K_1(ga_1) = K_1a_2 \implies \exists ~k \in K_1$ such that $(k \circ g)(a_1) = a_2$.\\
Let $g':= k \circ g \in G_1'$ then $g'(a_1) = a_2.$
Consider $g'L_1 \in G_1'/L_1$. Then $(g'L_1)(L_1a_1) = L_1a_2.$
This contradicts our assumption that $L_1a_1$ and $L_1a_2$ belong to two different $G_1'/L_1$-orbit of $M_1/L_1.$ This proves  Claim \ref{clm5}.\\
Let $R_1$ and $R_2$ denote the reflections of $E_1$ about $OA$ and $OA_1$ 
Now if the given map $X_1$ is $6$-semiregular then consider the group $G_4 = \langle \alpha_1, \beta_1, \chi_1, R_1\circ R_2 \rangle$ instead of $G_1'$. Then proceeding as above we get $Y_1 := E_1/L_3$ is a $2$-semiregular cover of $X_1$, where $L_3 = \langle \gamma^{m_2}, \delta^{m_2} \rangle \trianglelefteq G_4$ for some $m_2 \in \ZZ$. 
Now if the given map $X_1$ is $3$ or $2$ semiregular then consider the group $G_3 = \langle \alpha_1, \beta_1, \chi_1, R_1, R_2 \rangle$. Then proceeding as above we get $Y_2 := E_1/L_2$ is a $1$-semiregular cover of $X_1$, where $L_2 = \langle \gamma^{m_1}, \delta^{m_1} \rangle \trianglelefteq G_3$ for some $m_1 \in \ZZ$.

Now if two conjugate subgroups of Aut($E_1$) acts on $F(E_1)$ then they give same number of orbits. So to find how many different orbital maps are there for a given type we need to check number of orbits under action of a group taken from each conjugacy class of Aut($E_1$). 
Now we know that automorphism groups on the plane are of the form $T \rtimes S$ where $T$ is the translation group and $S$ is stabilizer of origin for the action of Aut($E_1$) on $E_1$. 
If $X=E_1/K$ is a toroidal map then Aut($X$)$=$Nor($K$)$/K$. 
Since $K$ contains only translations so Nor($K$) always contains $T\rtimes \langle \chi \rangle$.  
Thus one needs to determine which symmetries in $S$ normalizes $K$. Now $S=\langle R_1,R_2,\chi \rangle$. For type $[3^6]$ there are $4$ subgroups of $S$ up to conjugates. They are $\langle \chi \rangle, \langle \chi, R_2 \rangle, \langle \chi, R_1\circ R_2 \rangle, S$. Hence Aut($X$) is of the form $(T \rtimes K')/K$ where $K'$ is conjugate to one of the above groups. Thus it is enough to see number of orbits under action of $T \rtimes K'$ on $F(E_1)$. In above proof the groups $G_1',G_2,G_3$ are nothing but $T \rtimes K'$ for different $K'$. we did not get $4$ or $5$ orbits under action of these groups. Thus there does not exists $4$ or $5$-semiregular toroidal map of type $[3^6]$.
This completes the prove of part (a) of Theorem \ref{thm-main1}.

Now let $X_2$ be a semi-equivelar toroidal map of type $[4^4]$. Then by Proposition \ref{propo-1} we can assume that $X_2 = E_2/K_2$ for some fixed point free subgroup $K_2$ of Aut($X_2$). Thus $K_2$ consist of only translations and glide reflections. Since $X_2$ is orientable so $K_2$ contains only translations. Consider $H_2$, $K_2$ and $G_2$ as in proof of Theorem \ref{no-of-orbits}.  Now $F(X_2)$ has $4$ $G_2\mbox{-}$orbits. Consider the group $G_3 = \langle \alpha_2, \beta_2, \chi_2, \rho_{OA} \rangle$. Clearly $F(E_2)$ has $2$ $G_3$ orbits. Now proceeding in same way as in previous case we can prove that  $Y_2 := E_2/L_2$ is a $2$-semiregular cover of $X_2$, where $L_2 = \langle \gamma_2^{m_2}, \delta_2^{m_2} \rangle \trianglelefteq G_2$ for some $m_2 \in \ZZ$.
Now if the given map $X_2$ is $2$-semiregular then consider the group $G_3 = \langle \alpha_2, \beta_2, \chi_2, \rho_{OA}, \rho_{OB_1} \rangle$. Then proceeding as above we get $Y_1 := E_2/L_3$ is a $1$-semiregular cover of $X_2$, where $L_3 = \langle \gamma_2^{m_1}, \delta_2^{m_1} \rangle \trianglelefteq G_3$ for some $m_1 \in \ZZ$. 
There are $5$ subgroups of point stabilizer $S$ up to conjugates. They are $\langle \chi \rangle, \langle \chi, R_1 \rangle, \langle \chi, R_2 \rangle, \langle \chi,R_1\circ R_2 \rangle, S$.
Now by the same type of argument as in previous case we can see that there does not exists a $3$-semiregular map of type $[4^4]$. This completes the proof of part (b) in Theorem \ref{thm-main1}.
\end{proof}
Now we proof a series of lemmas and use them to proof other parts of Theorems   \ref{no-of-orbits} and \ref{thm-main1}.

For a given semi-equivelar map $X = E/K$ consider $\widehat{X}$ be the associated equivelar map defined by $\widehat{X} = \widehat{E}/K$, where $\widehat{E}$ be the associated equivelar tessellation obtained from $E$ as shown in Figures \ref{fig-E_1} to \ref{fig-E_4}.

\begin{lemma}\label{X-9}
Let $X_9 = E_{9}/G_9$ is semiregular toroidal map of type $[3^1,4^1,6^1,4^1]$. Then $\widehat{X_9}$ is $m_9$-semiregular if and only if $X_9$ is $4m_9$-semiregular.
\end{lemma}
\begin{proof}
Here by Theorem \ref{no-of-orbits} and \ref{thm-main1} we can conclude that  $m_4 \in \{1,2,3,6\}$. The case $m_4 = 1$ discussed in \cite{drach:2019}. Here we discuss $m_4 = 2,3$ and $6$.
Let $G_9 = \langle \alpha_9, \beta_9, \chi_9 \rangle$. Where $\alpha_9 : z \mapsto z+A_9$, $\beta_9 : z \mapsto z+B_9$ and $\chi_9$ be the $180$ degree rotation about origin, see Figure \ref{fig-E_9}. $\widehat{E_9}$ is of type $[3^6]$.
First suppose $m_9 = 6$.
Let $\widehat{X_9}$ is $6$-semiregular. ${\rm Aut}(\widehat{X_9}) = {\rm Nor_{Aut(\widehat{E_9})}}(K_9)/K_9 ={\rm Nor_{Aut(E_9)}}(K_9)/K_9 = {\rm Aut}(X_9) $. Now, $G_9 \le {\rm Nor_{Aut(\widehat{E_9})}}(K_9)$. Action of $G_9$ on $E(\widehat{E_9})$ also gives $6$ flag orbits. Hence $\widehat{X_9}$ to be $6$-semiregular we must have ${\rm Nor_{Aut(\widehat{E_9})}}(K_9) = G_9$ or some conjugate of $G_9$. 
Now under the action of $G_9$, $F(E_9)$ has $24$ orbits. Symmetries of $E_9$ which fixes origin are also symmetries of $\widehat{E_9}$. Hence $F(E_9)$ has $24$ ${\rm Nor_{Aut(E_9)}}(K_9)$-orbits. Thus $X_9$ is $24$-semiregular. \\
Now, let $m_9 = 3$. Then Aut($\widehat{X_9}$) is of the form $(H_9 \rtimes K')/K_9$ where $K'$ is conjugate to $\langle \chi_9,R_2 \rangle$ or $\langle \chi_9, R_1R_2 \rangle$. Since $m_9=3$ $K'$ is conjugate to $\langle \chi_9,R_2 \rangle$. One can see that under action of this group $F(X_9)$ has $12$ flag orbits. Thus $X_9$ is $12$-semiregular map.\\
Now, let $m_9 = 2$. Then Aut($\widehat{X_9}$) is of the form $(H_9 \rtimes K')/K_9$ where $K'$ is conjugate to $\langle \chi_9,R_2\circ R_1 \rangle$. One can see that under action of this group $F(E_9)$ has $8$ orbits. Thus $X_9$ is $8$-semiregular map.\\
Conversely, let $X_9$ is $24$ orbital. Then $G_9/K_9 \le $ Aut($X_9$). These symmetries are also present in ${\rm Aut}(\widehat{E_9})$ and the group $G_9$ gives $6$ orbits on $F(\widehat{E_9})$. Since, Aut($X_9$) does not change $G_9/K_9$-orbits of $F(X_9)$ so Aut($\widehat{X_9}$) will also not change $G_9/K_9$-orbits of $F(\widehat{X_9})$. Thus $\widehat{X_9}$ is $6$-orbital.
Now suppose $X_9$ is $12$ orbital. Then its automorphism group will contain either $R_1$ or $R_2$ along with $G_9$. With these symmetries $F(\widehat{X_9})$ will have $3$ orbits. Hence $\widehat{X_9}$ is $3$-semiregular.\\
Now suppose $X_9$ is $8$ orbital. Then its automorphism group will contain either $R_1$ and $R_2$ along with $G_9$. With these symmetries $F(\widehat{X_9})$ will have $2$ orbits. Hence $\widehat{X_9}$ is $2$-semiregular. This completes the proof of Lemma \ref{X-9}.
\end{proof}

\begin{lemma}\label{X-8}
Let $X_8 = E_{8}/G_8$ is semiregular toroidal map of type $[3^1,12^2]$. Then $\widehat{X_8}$ is $m_8$-semiregular if and only if $X_8$ is $3m_8$-semiregular.
\end{lemma}
\begin{proof}
Here by Theorem \ref{no-of-orbits} and \ref{thm-main1} we can conclude that  $m_8 \in \{1,2,3,6\}$. The case $m_8 = 1$ discussed in \cite{drach:2019}. Here we discuss $m_8 = 2,3$ and $6$.
Let $G_8 = \langle \alpha_8, \beta_8, \chi_8 \rangle$. Where $\alpha_8 : z \mapsto z+A_8$, $\beta_8 : z \mapsto z+B_8$ and $\chi_8$ be the $180$ degree rotation about origin, see Figure \ref{fig-E_8}. $\widehat{E_8}$ is of type $[3^6]$.
First suppose $m_8 = 6$.
Let $\widehat{X_8}$ is $6$-semiregular. Then by similar reason as in Lemma \ref{X-9}  we must have ${\rm Nor_{Aut(\widehat{E_8})}}(K_8) = G_8$ or some conjugate of $G_8$. 
Now under the action of $G_8$, $F(E_8)$ has $18$ orbits. Symmetries of $E_8$ which fixes origin are also symmetries of $\widehat{E_8}$. Hence $F(E_8)$ has $18$ ${\rm Nor_{Aut(E_8)}}(K_8)$-orbits. Thus $X_8$ is $18$-semiregular. \\
Now, let $m_8 = 3$. Then Aut($\widehat{X_8}$) is of the form $(H_8 \rtimes K')/K_8$ where $K'$ is conjugate to $\langle \chi_8,R_2 \rangle$ or $\langle \chi_8, R_1R_2 \rangle$. Since $m_8=3$ $K'$ is conjugate to $\langle \chi_8,R_2 \rangle$. One can see that under action of this group $F(X_8)$ has $9$ flag orbits. Thus $X_8$ is $9$-semiregular map.\\
Now, let $m_8 = 2$. Then Aut($\widehat{X_8}$) is of the form $(H_8 \rtimes K')/K_8$ where $K'$ is conjugate to $\langle \chi_8,R_2\circ R_1 \rangle$. One can see that under action of this group $F(E_8)$ has $6$ orbits. Thus $X_8$ is $6$-semiregular map.\\
Conversely, let $X_8$ is $18$-semiregular. Then $G_8/K_8 \le $ Aut($X_8$). These symmetries are also present in ${\rm Aut}(\widehat{E_8})$ and the group $G_8$ gives $6$ orbits on $F(\widehat{E_8})$. Since, Aut($X_8$) does not change $G_8/K_8$-orbits of $F(X_8)$ so Aut($\widehat{X_8}$) will also not change $G_8/K_8$-orbits of $F(\widehat{X_8})$. Thus $\widehat{X_8}$ is $6$-semiregular.
Now suppose $X_8$ is $9$-semiregular. Then its automorphism group will contain either $R_1$ or $R_2$ along with $G_8$. With these symmetries $F(\widehat{X_8})$ will have $3$ orbits. Hence $\widehat{X_8}$ is $3$-semiregular.\\
Now suppose $X_8$ is $6$-semiregular. Then its automorphism group will contain either $R_1$ and $R_2$ along with $G_8$. With these symmetries $F(\widehat{X_8})$ will have $2$ orbits. Hence $\widehat{X_8}$ is $2$-semiregular. This completes the proof of Lemma \ref{X-8}.
\end{proof}

\begin{lemma}\label{X-11}
Let $X_{11} = E_{11}/G_{11}$ is semiregular toroidal map of type $[4^1,6^1,12^1]$. Then $\widehat{X_{11}}$ is $m_{11}$-semiregular if and only if $X_{11}$ is $6m_{11}$-semiregular.
\end{lemma}
\begin{proof}
Here by Theorem \ref{no-of-orbits} and \ref{thm-main1} we can conclude that  $m_{11} \in \{1,2,3,6\}$. The case $m_{11} = 1$ discussed in \cite{drach:2019}. Here we discuss $m_{11} = 2,3$ and $6$.
Let $G_{11} = \langle \alpha_{11}, \beta_{11}, \chi_{11} \rangle$. Where $\alpha_{11} : z \mapsto z+A_{11}$, $\beta_{11} : z \mapsto z+B_{11}$ and $\chi_{11}$ be the $180$ degree rotation about origin, see Figure \ref{fig-E_11}. $\widehat{E_{11}}$ is of type $[3^6]$.
First suppose $m_{11} = 6$.
Let $\widehat{X_{11}}$ is $6$-semiregular. Then by similar reason as in Lemma \ref{X-9}  we must have ${\rm Nor_{Aut(\widehat{E_{11}})}}(K_{11}) = G_{11}$ or some conjugate of $G_{11}$. 
Now under the action of $G_{11}$, $F(E_{11})$ has $36$ orbits. Symmetries of $E_{11}$ which fixes origin are also symmetries of $\widehat{E_{11}}$. Hence $F(E_{11})$ has $36$ ${\rm Nor_{Aut(E_{11})}}(K_{11})$-orbits. Thus $X_{11}$ is $36$-semiregular. \\
Now, let $m_{11} = 3$. Then Aut($\widehat{X_{11}}$) is of the form $(H_{11} \rtimes K')/K_{11}$ where $K'$ is conjugate to $\langle \chi_{11},R_2 \rangle$ or $\langle \chi_{11}, R_1R_2 \rangle$. Since $m_{11}=3$ $K'$ is conjugate to $\langle \chi_{11},R_2 \rangle$. One can see that under action of this group $F(X_{11})$ has $18$ flag orbits. Thus $X_{11}$ is $18$-semiregular map.\\
Now, let $m_{11} = 2$. Then Aut($\widehat{X_{11}}$) is of the form $(H_{11} \rtimes K')/K_{11}$ where $K'$ is conjugate to $\langle \chi_{11},R_2\circ R_1 \rangle$. One can see that under action of this group $F(E_{11})$ has $12$ orbits. Thus $X_{11}$ is $12$-semiregular map.\\
Conversely, let $X_{11}$ is $36$-semiregular. Then $G_{11}/K_{11} \le $ Aut($X_{11}$). These symmetries are also present in ${\rm Aut}(\widehat{E_{11}})$ and the group $G_{11}$ gives $6$ orbits on $F(\widehat{E_{11}})$. By similar reason as above $\widehat{X_{11}}$ is $6$-semiregular.
Now suppose $X_{11}$ is $18$-semiregular. Then its automorphism group will contain either $R_1$ or $R_2$ along with $G_{11}$. With these symmetries $F(\widehat{X_{11}})$ will have $3$ orbits. Hence $\widehat{X_{11}}$ is $3$-semiregular.\\
Now suppose $X_{11}$ is $12$-semiregular. Then its automorphism group will contain either $R_1$ and $R_2$ along with $G_{11}$. With these symmetries $F(\widehat{X_{11}})$ will have $2$ orbits. Hence $\widehat{X_{11}}$ is $2$-semiregular. This completes the proof of Lemma \ref{X-11}.
\end{proof}

\begin{lemma}\label{X-5}
Let $X_5 = E_{5}/G_5$ is semiregular toroidal map of type $[3^2,4^1,3^1,4^1]$. If $\widehat{X_5}$ is $4$-semiregular then $X_5$ is $20$-semiregular.
\end{lemma}
\begin{proof}
Let $G_{5} = \langle \alpha_{5}, \beta_{5}, \chi_{5} \rangle$. Where $\alpha_{5} : z \mapsto z+A_{5}$, $\beta_{5} : z \mapsto z+B_{5}$ and $\chi_{5}$ be the $180$ degree rotation about origin, see Figure \ref{fig-E_5}. $\widehat{E_{5}}$ is of type $[4^4]$.
Let $\widehat{X_{5}}$ is $4$-semiregular. Then by similar reason as in Lemma \ref{X-9}  we must have ${\rm Nor_{Aut(\widehat{E_{5}})}}(K_{5}) = G_{5}$ or some conjugate of $G_{5}$. 
Now under the action of $G_{5}$, $F(E_{5})$ has $20$ orbits. Symmetries of $E_{5}$ which fixes origin are also symmetries of $\widehat{E_{5}}$. Hence $F(E_{5})$ has $20$ ${\rm Nor_{Aut(E_{5})}}(K_{5})$-orbits. Thus $X_{5}$ is $20$-semiregular.
\end{proof}

\begin{lemma}\label{X-6}
Let $X_6 = E_{6}/G_6$ is semiregular toroidal map of type $[4^1,8^2]$. Then $\widehat{X_6}$ is $4$-semiregular then $X_6$ is $12$-semiregular. 
\end{lemma}
\begin{proof}
Let $G_{6} = \langle \alpha_{6}, \beta_{6}, \chi_{6} \rangle$. Where $\alpha_{6} : z \mapsto z+A_{6}$, $\beta_{6} : z \mapsto z+B_{6}$ and $\chi_{6}$ be the $180$ degree rotation about origin, see Figure \ref{fig-E_6}. $\widehat{E_{6}}$ is of type $[4^4]$.
Let $\widehat{X_{6}}$ is $4$-semiregular. Then by similar reason as in Lemma \ref{X-9}  we must have ${\rm Nor_{Aut(\widehat{E_{6}})}}(K_{6}) = G_{6}$ or some conjugate of $G_{6}$. 
Now under the action of $G_{6}$, $F(E_{6})$ has $12$ orbits. Symmetries of $E_{6}$ which fixes origin are also symmetries of $\widehat{E_{6}}$. Hence $F(E_{6})$ has $12$ ${\rm Nor_{Aut(E_{6})}}(K_{6})$-orbits. Thus $X_{6}$ is $12$-semiregular.
\end{proof}

\begin{lemma}\label{X-7}
Let $X_{7} = E_{7}/G_{7}$ is semiregular toroidal map of type $[3^1,6^1,3^1,6^1]$. Then $\widehat{X_{7}}$ is $m_{7}$-semiregular if and only if $X_{7}$ is $4m_{7}$-semiregular.
\end{lemma}
\begin{proof}
Here by Theorem \ref{no-of-orbits} and \ref{thm-main1} we can conclude that  $m_{7} \in \{1,2,3,6\}$. The case $m_{7} = 1$ discussed in \cite{drach:2019}. Here we discuss $m_7 = 2,3$ and $6$.
Let $G_7 = \langle \alpha_7, \beta_7, \chi_7 \rangle$. Where $\alpha_7 : z \mapsto z+A_7$, $\beta_7 : z \mapsto z+B_7$ and $\chi_7$ be the $180$ degree rotation about origin, see Figure \ref{fig-E_7}. $\widehat{E_7}$ is of type $[3^6]$.
First suppose $m_7 = 6$.
Let $\widehat{X_7}$ is $6$-semiregular. Then by similar reason as in Lemma \ref{X-9}  we must have ${\rm Nor_{Aut(\widehat{E_7})}}(K_7) = G_7$ or some conjugate of $G_7$. 
Now under the action of $G_7$, $F(E_7)$ has $24$ orbits. Symmetries of $E_7$ which fixes origin are also symmetries of $\widehat{E_7}$. Hence $F(E_7)$ has $24$ ${\rm Nor_{Aut(E_7)}}(K_7)$-orbits. Thus $X_7$ is $24$-semiregular. \\
Now, let $m_7 = 3$. Then Aut($\widehat{X_7}$) is of the form $(H_7 \rtimes K')/K_7$ where $K'$ is conjugate to $\langle \chi_7,R_2 \rangle$ or $\langle \chi_7, R_1R_2 \rangle$. Since $m_7=3$ $K'$ is conjugate to $\langle \chi_7,R_2 \rangle$. One can see that under action of this group $F(X_7)$ has $12$ flag orbits. Thus $X_7$ is $12$-semiregular map.\\
Now, let $m_7 = 2$. Then Aut($\widehat{X_7}$) is of the form $(H_7 \rtimes K')/K_7$ where $K'$ is conjugate to $\langle \chi_7,R_2\circ R_1 \rangle$. One can see that under action of this group $F(E_7)$ has $8$ orbits. Thus $X_7$ is $8$-semiregular map.\\
Conversely, let $X_7$ is $24$-semiregular. Then $G_7/K_7 \le $ Aut($X_7$). These symmetries are also present in ${\rm Aut}(\widehat{E_7})$ and the group $G_7$ gives $6$ orbits on $F(\widehat{E_7})$. By similar reason as above $\widehat{X_7}$ is $6$-semiregular.
Now suppose $X_7$ is $12$-semiregular. Then its automorphism group will contain either $R_1$ or $R_2$ along with $G_7$. With these symmetries $F(\widehat{X_7})$ will have $3$ orbits. Hence $\widehat{X_7}$ is $3$-semiregular.\\
Now suppose $X_7$ is $8$-semiregular. Then its automorphism group will contain either $R_1$ and $R_2$ along with $G_7$. With these symmetries $F(\widehat{X_7})$ will have $2$ orbits. Hence $\widehat{X_7}$ is $2$-semiregular. This completes the proof of Lemma \ref{X-7}.
\end{proof}

\begin{lemma}
Let $X_{10} = E_{10}/G_{10}$ is semiregular toroidal map of type $[3^4,6^1]$. Then $\widehat{X_{10}}$ is $6$-semiregular if and only if $X_9$ is $30$-semiregular.
\end{lemma}
\begin{proof}
Let $G_{10} = \langle \alpha_{10}, \beta_{10}, \chi_{10} \rangle$. Where $\alpha_{10} : z \mapsto z+A_{10}$, $\beta_{10} : z \mapsto z+B_{10}$ and $\chi_{10}$ be the $180$ degree rotation about origin, see Figure \ref{fig-E_10}. $\widehat{E_{10}}$ is of type $[3^6]$.
Let $\widehat{X_{10}}$ is $6$-semiregular. Then by similar reason as in Lemma \ref{X-9}  we must have ${\rm Nor_{Aut(\widehat{E_{10}})}}(K_{10}) = G_{10}$ or some conjugate of $G_{10}$. 
Now under the action of $G_{10}$, $F(E_{10})$ has $30$ orbits. Symmetries of $E_{10}$ which fixes origin are also symmetries of $\widehat{E_{10}}$. Hence $F(E_{10})$ has $30$ ${\rm Nor_{Aut(E_{10})}}(K_{10})$-orbits. Thus $X_{10}$ is $30$-semiregular. \\
Conversely, let $X_{10}$ is $30$-semiregular. Then $G_{10}/K_{10} \le $ Aut($X_{10}$). These symmetries are also present in ${\rm Aut}(\widehat{E_{10}})$ and the group $G_{10}$ gives $6$ orbits on $F(\widehat{E_{10}})$. By similar reason as above $\widehat{X_{10}}$ is $6$-semiregular.
\end{proof}

\begin{proof}[Proof of Theorem \ref{no-of-orbits} continued]
Here we will show that the bounds in Theorem \ref{no-of-orbits} are strict. 
Now we show that there existence of a $6$-semiregular toroidal map of type $[3^6]$.
Let $X$ be a equivelar map of type $[3^6]$. Then $X=E_1/K$ for some discrete fixed point free subgroup $K$ of Aut($E_1$). Aut($X$)$=$ Nor($K$)/$K$. Now $F(X)$ has $6$ $G_1/K$-orbits. If we can show that there exists some $K \le H_1$ such that Nor($K$)$=G_1$ then we are done.\\
Consider $K=\langle \alpha_1^5,\beta_1^3 \rangle$. $\alpha_1^5$ and $\beta_1^3$ are translations by the vectors $5A_1$ and $3B_1$ respectively.\\
${\rm Nor}(K)=\{\gamma \in {\rm Aut}(E_1) \mid \gamma\alpha_1^5\gamma^{-1},\gamma\beta_1^3\gamma^{-1}\in K \} = \{\gamma \in {\rm Aut}(E_1) \mid \gamma(5A_1), \gamma(3B_1) \in \ZZ5A_1+\ZZ3B_1 \}$. Clearly $G_1 \le {\rm Nor}(K)$. But $60$ and $120$ degree rotations and reflection about a line does not belongs to ${\rm Nor}(K)$. Hence ${\rm Nor}(K) = G_1$. The same process will work for equivelar maps of type $[4^4]$.
For other semi-equivelar maps we use above lemmas to conclude the bounds are sharp. The argument will go as following.
Let $X$ be a semi-equivelar map of type $[p_1^{r_1}, p_2^{r_2}, \dots, p_k^{r_k}]$ such that $\widehat{X}$ is of type $[3^6]$. Now by above discussion there exists a $6$ semiregular toridal map of type $[3^6]$. Now using above lemmas depending on type of $X$ it follows that the bounds of flag orbits are strict. Similarly we can do for maps whose corresponding equivelar map is of type $[4^4]$. This completes the proof of Theorem \ref{no-of-orbits}. 
\end{proof}

\begin{proof}[Proof of Theorem \ref{thm-main1} continued]
Let $X_9$ be a $m_9$-semiregular toroidal map of type $[3^1,4^1,6^1,4^1]$. Let $\widehat{X_9}$ be the associated equivelar map of type $[3^6]$. By Lemma \ref{X-9} we get $\widehat{X_3}$ has $n_9:=m_9/2$ many edge  orbits. Now by Theorem \ref{thm-main1} we have covering $\eta_{k_9} : \widehat{Y_{k_9}} \to \widehat{X_9}$ where $\widehat{Y_{k_9}}$ is $k_9$-semiregular for each $ (k_9,n_9) = (1,2),(2,6),(1,3),(3,6),(1,6)$. Now, if we consider the map of type $[3^1,4^1,6^1,4^1]$ corresponding to the equivelar map $\widehat{Y_{k_9}}$, say $Y_{k_9}$, then by Lemma \ref{X-9} it will be a $(4 \times k_9)$-edge orbital map. Clearly $Y_{k_9}$ is a cover of $X_9$. Hence for given $m_9$-semiregular map of type $[3^1,4^1,6^1,4^1]$ there exists a $k_9$ orbital cover of it for each $(k_9,m_9)=(4,8),(8,24),(4,12),(12,24),(4,24)$.\\
Proceeding in exactly similar way we can conclude the followings also. 
Given $m_7$-semiregular map of type $[3^1,6^1,3^1,6^1]$ there exists a $k_7$-semiregular cover of it for each $(k_7,m_7)=(4,8),(8,24),(4,12),(12,24),(4,24)$.\\
Given $m_8$-semiregular map of type $[3^1,12^2]$ there exists a $k_8$-semiregular cover of it for each $(k_8,m_8)=(3,6),(3,9),$ $(3,18),(6,18),(9,18)$.\\
Given $m_{11}$-semiregular map of type $[4^1,6^1,12^1]$ there exists a $k_{11}$-semiregular cover of it for each $(k_{11},m_{11})=(6,12),(6,18),(6,36),(12,36),(18,36)$.\\
Given $m_{6}$-semiregular map of type $[4^1,8^2]$ there exists a $k_{6}$-semiregular cover of it for each $(k_{6},m_{6})=(6,12),(3,6),(3,12)$.This completes the proof of parts (c),(d),(e),(f),(g) of Theorem \ref{thm-main1}.\\
Let $X_4$ be a $10$-semiregular map. We can take $X_4 = E_4/K_4$ for some $K_4 \le H_4 \le {\rm Aut}(E_4)$. Let $G_4$ be as in the proof of Theorem \ref{no-of-orbits}. Then $F(E_4)$ has $10$ flag orbits. Consider $G_4' = \langle G_4, R_1 \rangle$ where $R_1$ is the map obtained by taking reflection of $E_4$ about the line passing through $O$ and $A$ (see Figure \ref{fig-E_4}). Observe that $F(E_4)$ has $5$ $G_4'$-orbits. Now proceeding in similar way as in part (a) of this theorem we get existence of a $5$-semiregular cover of $X_4$. This proves part (h) of Theorem \ref{thm-main1}.\\
Let $X_5$ be a $20$-semiregular map of type $[3^2,4^1,3^1,4^1]$. We can take $X_5 = E_5/K_5$ for some $K_5 \le H_5 \le {\rm Aut}(E_5)$. Let $G_5$ be as in the proof of Theorem \ref{no-of-orbits}. Then $F(E_5)$ has $20$ flag orbits. Consider $G_5' = \langle \alpha_5,\beta_5,\chi_5, R_1 \rangle$ and $G_5''=\langle \alpha_5,\beta_5,\chi_5, R_1,R_2 \rangle$ where $R_1$ and $R_2$ is the map obtained by taking reflection of $E_5$ about the line passing through $O$,$A$ and $A_5,B_5$ respectively (see Figure \ref{fig-E_5}). Observe that $F(E_5)$ has $10$ $G_4'$-orbits and $5$ $G_4''$-orbits. Now proceeding in similar way as in part (a) of this theorem we get existence of a $10$ and $5$-semiregular cover of $X_5$. This proves part (i) of Theorem \ref{thm-main1}.\\
Let $X_{10}$ be a $30$-semiregular map. We can take $X_{10} = E_{10}/K_{10}$ for some $K_{10} \le H_{10} \le {\rm Aut}(E_{10})$. Let $G_{10}$ be as in the proof of Theorem \ref{no-of-orbits}. Then $F(E_{10})$ has $30$ flag orbits. Consider $G_{10}' = \langle \alpha_{10},\beta_{10},\chi_{10},\rho_{10} \rangle$ where $R_1$ is the map obtained by taking rotation of $E_{10}$ about origin (see Figure \ref{fig-E_10}). Observe that $F(E_{10})$ has $10$ $G_{10}'$-orbits. Now proceeding in similar way as in part (a) of this theorem we get existence of a $10$-semiregular cover of $X_{10}$. This proves part (j) of Theorem \ref{thm-main1}.
\end{proof}

Now we are moving to see number of sheets of the covers obtained above. For that we make,

\begin{claim}\label{sheet} 
The cover $Y$ in Theorem \ref{thm-main1} is a $m^2$ sheeted covering of $X$.
\end{claim}

To do this we need following two results from the theory of covering spaces.

\begin{result}\label{result1} (\cite{AH2002})
Let $p:(\widetilde{X},\widetilde{x_0})\to (X,x_0)$ be a path-connected covering space of the path-connected, locally path-connected space $X$, and let $H$ be the subgroup $p_*(\pi_1(\widetilde{X},\widetilde{x_0})) \subset \pi_1(X,x_0).$ Then,
\begin{enumerate}
    \item[1.] This covering space is normal if and only if $H$ is a normal subgroup of $\pi_1(X,x_0)$
    \item[2.] $G(\widetilde{X})$(the group of deck transformation of the covering $\widetilde{X}\to X$) is isomorphic to $N(H)/H$ where $N(H)$ is the normalizer of $H$ in $\pi_1(X,x_0)$.
\end{enumerate}
In particular, $G(\widetilde{X})$ is isomorphic to $\pi_1(X,x_0)/H$ if $\widetilde{X}$ is a normal covering. Hence for universal cover $\widetilde{X}\to X$ we have $G(\widetilde{X}) \simeq \pi_1(X)$.
\end{result} 

\begin{result}\label{result2} (\cite{AH2002})
The number of sheets of a covering space $p:(\widetilde{X},\widetilde{x_0})\to (X,x_0)$ with $X$ and $\widetilde{X}$ path-connected equals the index of $p_*(\pi_1(\widetilde{X},\widetilde{x_0}))$ in $\pi_1(X,x_0)$.
\end{result}

In our situation applying Result \ref{result1} for the covering $E_i\to E_i/K_i$ we get $\pi_1(E_i/K_i) = K_i$. For the covering $E_i \to E_i/\mathcal{L}_i$ we get $\pi_1(E_i/\mathcal{L}_i) = \mathcal{L}_i$. Thus applying Result \ref{result2} we get number of sheets of $Y$ over $X$ is $=n:=[K_i:\mathcal{L}_i] = m^2$ for all $i = 3,4,5,6, 7$.  This proves our Claim \ref{sheet}.

\begin{proof}[Proof of Theorem \ref{thm-main2}]
Let $X$ be an semiregular map of type $(m, \ell, u,v)$. Then form Prop. \ref{propo-1} we get $X = M_i/K$ for some discrete subgroup $K$ of Aut($M_i$). Now $Y$ covers $X$ if and only if  $Y = M_i/L$ for some subgroup $L$ of $K$ generated by $2$ translations corresponding to $2$ independent vectors. Let $K = \langle \gamma , \delta \rangle$. Now consider $L_n = \langle \gamma^n, \delta \rangle$ and $Y_n = M_i/L_n$. then $Y_n$ covers $X$.
Number of sheets of the cover $Y_n \longrightarrow X$ is equal to $[K : L_n] = n.$ Hence $Y_n$ is our required $n$ sheeted cover of $X$.
\end{proof}
\begin{proof}[Proof of Theorem \ref{thm-main3}]
Here two maps are isomorphic if they are isomorphic as maps. Two maps are equal if the orbits of $\mathbb{R}^2$ under the action of corresponding groups are equal as sets.
Suppose $X$ and $K$ be as in the proof of Theorem \ref{thm-main2}. Let $n \in \mathbb{N}$. 
Let $Y =  E/L$ be $n$ sheeted cover of $X$. 
Let $L = \langle \omega_1, \omega_2\rangle$. Where $\omega_1, \omega_2 \in K = \langle \gamma, \delta \rangle$. Suppose $\omega_1 = \gamma^a \circ \delta^b$ and $\omega_2 = \gamma^c \circ \delta^d$ where $a,b,c,d \in \ZZ$. Define $M_Y = \begin{bmatrix} a & c\\b &d \end{bmatrix}$.
We represent $Y$ by the associated matrix $M_Y$. This matrix representation corresponding to a map is unique as $\gamma$ and $\delta$ are translations along two linearly independent vectors. Denote area of the torus $Y$ by $\Delta_Y$. As $Y$ is $n$ sheeted covering of $X$ so $\Delta_Y = n\Delta_X \implies$ area of the parallelogram spanned by $w_1$ and $w_2 = n \times$ area of the parallelogram spanned by $\gamma$ and $\delta$. That means $|det(M_Y)| = n$. Therefore for each $n$ sheeted covering, the associated matrix belongs to 
$$ \mathcal{S}:=\{ M \in GL(2,\ZZ): |det(M)| = n \}.$$ 
Conversely for every element of $\mathcal{S}$ we get a $n$ sheeted covering $Y$ of $X$ by associating $\begin{bmatrix} a & c\\b &d \end{bmatrix}$ to $E/\langle a\gamma + b\delta, c\gamma +d\delta\rangle$. So there is an one to one correspondence to $n$-sheeted covers of $X$ and $\mathcal{S}$. To proceed further we need following two lemmas.
\begin{lemma}\label{equl}
Let $Y_1$ and $Y_2$ be maps and $M_1$ and $M_2$ be associated matrix of them respectively. Then $Y_1 = Y_2$ if and only if  there exists an unimodular matrix $($an integer matrix with determinant $1$ or $-1)$ $U$ such that $M_1U = M_2.$  
\end{lemma}
\begin{proof}
Let $Y_1 = Y_2$. Let $i:Y_1\to Y_2$ be an isomorphism. We can extend  $i$ to $\widetilde{i} \in Aut(E)$. Then $\widetilde{i}$ will take fundamental parallelogram of $Y_1$ to that of $Y_2$. Hence the latices formed by $L_1$ and $L_2$ are same say $\Lambda$. $\widetilde{i}$ transforms $\Lambda$ to itself. Therefore from \cite{HW1979}(Theorem 32, Chapter 3) we get matrix of the transformation is unimodular. Our lemma follows from this.
\\
Conversely suppose $M_1U = M_2$ where $U$ is an unimodular matrix. Let $M_1 = (w_1~ w_2), M_2 = (w_1'~w_2') $ and $U = \begin{bmatrix} a & b\\c &d \end{bmatrix}$ where $w_i, w_i'$ are column vectors for $i=1,2$.
Therefore 
$$ M_2 = M_1U \implies (w_1'~ w_2') = (w_1~ w_2)\begin{bmatrix} a & b\\c & d \end{bmatrix} = (aw_1+cw_2~~bw_1+dw_2).$$
Now suppose $L_1 = \langle \alpha_1, \beta_1 \rangle$ and $L_2 = \langle \alpha_2, \beta_2 \rangle$ and $A_i, B_i$ be the vectors by which $\alpha_i$ and $\beta_i$ translating the plane for $i=1,2$ and let $C$ and $D$ be the vectors corresponding to $\gamma$ and $\delta$.
Let $$A_1 = p_1C + q_1D, B_1 = s_1C + t_1D, A_2 = p_2C + q_2D, B_2 = s_2C + t_2D.$$
Now $w_1' = \begin{pmatrix} p_2 \\ q_2\end{pmatrix} = a\begin{pmatrix} p_1 \\ q_1\end{pmatrix} + c\begin{pmatrix} s_1 \\ t_1\end{pmatrix} = \begin{pmatrix} ap_1+cs_1 \\ aq_1+ct_1\end{pmatrix}$.
Therefore 
\begin{equation*}
    \begin{split}
        A_2 &= (ap_1 + cs_1)C + (aq_1+ct_1)D \\&= a(p_1C + q_1D) + c(s_1C + t_1D) \\&= aA_1 + cB_1
    \end{split}
\end{equation*}
Hence $\alpha_2 \in L_1$. Similarly $\beta_2 \in L_1$. Therefore $L_2 \le L_1$. Proceeding in the similar way and using the fact that $det(U) = \pm 1$ we get $L_1 \le L_2$.
Therefore $L_1 = L_2$. Thus $Y_1 = E/L_1 = E/L_2 = Y_2$. This completes the proof of Lemma \ref{equl}.
\end{proof}
\begin{lemma}\label{isomm}
Let $Y_1$ and $Y_2$ be two toroidal maps with associated matrix $M_1$ and $M_2$ respectively. Then $Y_1 \simeq Y_2$ if and only if  there exists $A \in G_0$ and $B\in GL(2,\ZZ)$ such that $M_1 = AM_2B$ where $G_0$ is group of rotations and reflections fixing the origin in $E$.
\end{lemma}
\begin{proof}
Let $Y_1 \simeq Y_2$ and $\alpha : Y_1 \to Y_2$ be an isomorphism. Now $\alpha$ can be extended to an automorphism of the covering plane $E$, call that extension be $\widetilde{\alpha}$. 
Clearly $\widetilde{\alpha}$ will take fundamental parallelogram of $Y_1$ to that of $Y_2$. Now the only ways to transform one fundamental region to another are rotation, reflection and change of basis of $E$. Multiplication by an element of $GL(2,\ZZ) $ will take care of base change. Rotation, reflection or their composition will take care by multiplication by $A\in G_0$. Hence we get $M_1 = AM_2B$.\\
Conversely let $M_1 = AM_2B$. $A \in G_0$ so the combinatorial type of the torus associated to the matrix $AM_2$ and $M_2$ are same. Geometrically multiplying by elements of $GL(2,\ZZ)$ corresponds to modifying the fundamental domain by changing the basis. Hence this will not change the combinatorial type of the torus. Thus $Y_1 \simeq Y_2$. This completes the proof of Lemma \ref{isomm}.
\end{proof}
Now define a relation on $\mathcal{S}$ by $P\sim Q \iff P= QU$ for some unimodular matrix $U$. Clearly this is an equivalence relation. Consider $\mathcal{S}' = \mathcal{S}/\sim $. So by Lemma \ref{equl} we can conclude that there are $\#\mathcal{S}'$ many distinct $n$ sheeted cover of $X$ exists. Let's find this cardinality.
Now for every $m\times n$ matrix $P$ with integer entries has an unique $m \times n$ matrix $H$, called hermite normal form of $P$, such that $H = PU$ for some unimodular matrix $U$. All elements of an equivalence class of $\mathcal{S}'$ has  same hermite normal form and we take this matrix in hermite normal form as representative of that equivalence class. Thus to find cardinality of $\mathcal{S}'$ it is enough to find number of distinct matrices $M$ which are in hermite normal form and has determinant $n$. We do not take the matrices with determinant $-n$ because by multiplying by the unimodular matrix  
$\begin{bmatrix} 0 & 1\\1 &0 \end{bmatrix}$ changes sign of the determinant.
As $M$ is in lower triangular form so take $M = \begin{bmatrix} a & 0\\b &d \end{bmatrix}$.
Then $det(M) = ad = n \implies a= n/d.$ By definition of hermite normal form $b\geq0$ and $b<d$ so $b$ has $d$ choices for each $d|n.$ Hence there are precisely $\sigma(n):=\sum_{d|n}d$ many distinct $M$ possible. Thus $\#\mathcal{S}' = \sigma(n).$
Let 
$\mathcal{S}_1 = \{ M| M $ is a representative of an equivalence class of $ \mathcal{S}'$ which is in hermite normal form$\}$
Clearly $\#\mathcal{S}_1 = \sigma(n).$ Now define a relation on $\mathcal{S}_1 $ by $M_1\sim M_2  \iff \exists A \in G_0 ~such ~that~ M_1 = AM_2$. Clearly this is an equivalence relation. 
Consider $\mathcal{S}_2 := \mathcal{S}_1/\sim$.
By Lemma \ref{isomm} it follows that there are $\#\mathcal{S}_2$ many $n$ sheeted covers upto isomorphism. Because here all matrices $M_i$ has same determinant so $M_1 = AM_2B \implies det(A)det(B) = 1$. As det($A$) and det($B$) both are integer so they belongs to $\{1,-1\}$ i.e. they are unimodular matrices.
Now we have to find $\#\mathcal{S}_2$.  Observe that the matrix representation of elements of $G_0$ with respect to the basis $\{\alpha(0), \beta(0)\}$ have integer entries because lattice points must go to lattice points by a symmetry of the plane where $\alpha : z \mapsto z+A_i$ and $\beta:z\mapsto z+B_i$ are two translations of $E_i$. Suppose $M_1$ and $M_2 \in \mathcal{S}_2$ such that $M_1 \sim M_2$. So there exists $A \in G_0$ such that $M_1 = AM_2$. Let 
$M_1 = \begin{bmatrix} \frac{n}{d_1} & 0\\c_1 &d_1 \end{bmatrix}$ , 
$M_2 = \begin{bmatrix} \frac{n}{d_2} & 0\\c_2 &d_2 \end{bmatrix}$ and 
$A = \begin{bmatrix} p & q\\r &s \end{bmatrix}$. 
Then 
\begin{equation}\label{mat}
    \begin{split}
        M_1 = AM_2 &\implies  \begin{bmatrix} \frac{n}{d_1} & 0\\c_1 &d_1 \end{bmatrix} = \begin{bmatrix} p & q\\r &s \end{bmatrix} \begin{bmatrix} \frac{n}{d_2} & 0\\c_2 &d_2 \end{bmatrix} = \begin{bmatrix} \frac{np}{d_2}+qc_2 & qd_2\\\frac{rn}{d_2}+sc_2 &sd_2 \end{bmatrix} \\
        &\implies qd_2 = 0 \\&\implies q=0 ~since ~d_2 \neq 0.
    \end{split}
\end{equation}
Therefore $A = \begin{bmatrix} p & 0\\r &s \end{bmatrix}$. $det(A) = 1 \implies ps = 1 \implies s= 1/p.$
Again from equation \ref{mat} we get $$np/d_2 = n/d_1 \implies p = d_2/d_1$$ and $$rn/d_2 + c_2/p = c_1 \implies r = (d_2c_1-d_1c_2)/n$$
Therefore $$A = \begin{bmatrix} d_2/d_1 & 0\\(d_2c_1-d_1c_2)/n &d_1/d_2 \end{bmatrix}.$$ As $A$ has integer entries and $d_1, d_2$ are positive so $d_1 = d_2 = d$(say) and $n|d(c_1-c_2)$.
Hence $$A = \begin{bmatrix} 1 & 0\\ d(c_1-c_2)/n & 1 \end{bmatrix}.$$
Now $A \in G_0$ and $G_0 = D_6$ for maps of type $[3^6],[6^3], [3^1,6^1,3^1,6^1],$, $[3^1,4^1,6^1,4^1], [3^1,12^2],$ $[4^1,6^1,12^1]$; $G_0 = D_4$ for maps of type $[4^4],[3^2,4^1,3^1,4^1],[4^1,8^2]$; $G_0=\ZZ_6$ for $[3^4,6^1]$; $G_0=\ZZ_2^2$ for maps of type $[3^3,4^2]$.
Here $D_6$ is generated by $\begin{bmatrix} 0 & -1\\ 1 & 1 \end{bmatrix}$ and $\begin{bmatrix} -1 & -1\\ 0 & 1 \end{bmatrix}$. $D_4$ is generated by $\begin{bmatrix} 0 & 1\\ -1 & 0 \end{bmatrix}$ and $\begin{bmatrix} -1 & 0\\ 0 & 1 \end{bmatrix}$. $\ZZ_6$ is generated by $\begin{bmatrix} 0 & -1\\ 1 & 1 \end{bmatrix}$.
One can check that only matrices in $G_0$ having diagonal entries $1$ is identity matrix. Hence $A = I_2$. Therefore $c_1 = c_2 \implies M_1 = M_2.$ Each equivalence class of $\mathcal{S}_2$ is singleton. Therefore $\#\mathcal{S}_2 = \#\mathcal{S}_1 = \sigma(n)$. This proves  Theorem \ref{thm-main3}.
\end{proof}

\begin{proof}[Proof of Theorem \ref{thm-main4}]
Let $X$ be a $m$-orbital map of vertex type $(m,  \ell; u,v)$. Let $Y_1$ be a $k$-orbital cover of $X$. Consider number of sheets of the cover $Y_1\longrightarrow X$ be $n_1$. Let the set $C_1$ containing all $n$ sheeted covering of $X$ for $n\le n_1-1$. Now check that does there exists a $k$-orbital cover or not in $C_1$. If there does not exists one, then $Y_1$ be a minimal $k$-orbital cover otherwise take $Y_2$ be a $k$-orbital cover in $C_1$. Let number of sheets for the covering $Y_2\longrightarrow X$ be $n_2$. Then consider $C_2$ be the collection of all $s$ sheeted cover of $X$ for $s\le n_2$. Again check if there exists a $k$-orbital cover in $C_2$. If not then $Y_2$ minimal $k$-orbital cover of $X$. Otherwise proceed similarly to more lower sheeted covering. Since there are only finitely many covers of each sheeted so the process will terminate. This proves  Theorem \ref{thm-main4}
\end{proof}
Now to answer of the last part of Question \ref{ques} we prove the following,
\begin{claim}\label{lem1}
Let $X$ be a $m$-orbital map. Then there exists a group $\widetilde{G} \le$ Aut($E$) such that $E(E)$ has $m$ $\widetilde{G}$-orbits.
\end{claim}
\begin{proof}
Let $X$ be a semi-equivelar toroidal map of type $[p_1^{r_1},p_2^{r_2},\dots p_k^{r_k}]$. By proposition \ref{propo-1} we get $X = E_j/K$ for some discrete subgroup $K$ of Aut($E_j$) where $E_j$ is semi-equivelar tilling of $\mathbb{R}^2$. 
Let $O_1,O_2, \dots O_m$ be $G$-orbits of $E(X)$. Let $\eta:E_j\to X$ be the covering map. Then $\{\eta^{-1}(O_i) | i = 1,2, \dots m\}$ be a partition of $E(E_j)$. 
Aut($X$)$=$Nor($K$)/$K$. 
Now consider $\widetilde{G} = {\rm Nor}(K).$ Then
$E(E_j)$ forms $m$ $\widetilde{G}$-orbits. This proves  Claim \ref{lem1}.
\end{proof}
\begin{lemma}\label{orbb}
Let $X$ be a $m$-orbital semiregular toroidal map and $Y$ be a $k$-orbital cover of $X$. Then $k\le m$.
\end{lemma}
\begin{proof}
Let $k \geq m+1$. Let $O_1, O_2, \dots, O_{m+1}$
be distinct Aut($Y$)-orbits of $E(Y)$. Let $\eta$ be the covering map. Suppose $a_i \in O_i$ for $i = 1,2, \dots m+1$. Then $\eta(a_i) \in E(X) ~\forall i$. Since $E(X)$ has $m$ orbits so by pigeon hole principle there exists $i,j \in \{1,2,3, \dots m+1\}$ such that $\eta(a_i), \eta(a_j)$ are in same Aut($X$) orbits of $E(X)$. Therefore there exists  $\Upsilon \in$ Aut($X$) such that $\Upsilon(\eta(a_i)) = \eta(a_j)$. Let $\widetilde{\Upsilon}\in$ Aut($Y$) be the preimage of $\upsilon$ under the projection $p:{\rm Aut}(Y)\to {\rm Aut}(X)$. If $a_i$ and $a_j$ belongs to same sheet of the covering $Y\longrightarrow X$ then $\widetilde{\Upsilon}(a_i) = a_j$. If $a_i$ and $a_j$ belongs to two different sheet then apply a suitable translation on $a_j$ and get an element $a_j' \in O(a_j)$ such that $a_i$ and $a_j'$ belongs to same sheet. Therefore in both cases $ \exists ~\widetilde{\Upsilon} \in$ Aut($Y$) such that $\widetilde{\Upsilon}(a_i) = a_j$. This is a contradiction to $a_i$ and $a_j$ are in different orbits. This proves Lemma \ref{orbb}.
\end{proof}

\section{Acknowledgements}

Authors are supported by NBHM, DAE (No. 02011/9/2021-NBHM(R.P.)/R$\&$D-II/9101).


{\small

}

\end{document}